\documentclass[12pt]{article} 
\usepackage{amsfonts,amsmath,latexsym,amssymb,mathrsfs,amsthm,comment,setspace}
\usepackage{slashbox}
\usepackage{caption}

\let\OLDthebibliography\thebibliography
\renewcommand\thebibliography[1]{
  \OLDthebibliography{#1}
  \setlength{\parskip}{0pt}
  \setlength{\itemsep}{0pt plus 0.0ex}
}

\newcommand{\bd}{\boldsymbol}

\newcommand{\R}{\mathbb{R}}

\def\numberlikeadb{\global\def\theequation{\thesection.\arabic{equation}}}
\numberlikeadb
\newtheorem{theorem}{Theorem}[section]
\newtheorem{lemma}[theorem]{Lemma}
\newtheorem{corollary}[theorem]{Corollary}

\newtheorem{proposition}[theorem]{Proposition}
\newtheorem{remark}[theorem]{Remark}

\newtheorem{conjecture}[theorem]{Conjecture}

\usepackage{color}

\usepackage{soul}

\usepackage{lscape}
\usepackage{caption}
\usepackage{multirow}
\begin{document}

\title{On bounds for the mode and median of the generalized hyperbolic and related distributions\footnote{J. Math. Anal. Appl. $\mathbf{493}$ (2021) Article 124508, https://doi.org/10.1016/j.jmaa.2020.124508.}}
\author{Robert E. Gaunt\footnote{Department of Mathematics, The University of Manchester, Oxford Road, Manchester M13 9PL, UK, robert.gaunt@manchester.ac.uk}\:\, and Milan Merkle\footnote{School of Electrical Engineering, University of Belgrade, Belgrade, Serbia, emerkle@etf.rs}}

\date{}
\maketitle

\vspace{-5mm}


\begin{abstract}

Except for certain parameter values, a closed form formula for the mode of the generalized hyperbolic (GH) distribution is not available.  In this paper, we exploit results from the literature on modified Bessel functions and their ratios to obtain simple but tight two-sided inequalities for the mode of the GH distribution for general parameter values.  As a special case, we deduce tight two-sided inequalities for the mode of the variance-gamma (VG) distribution,
and through a similar approach we also obtain tight two-sided inequalities for the mode of the  McKay Type I distribution.  The analogous problem for the median is more challenging, but we conjecture some monotonicity results for the median of the VG and  McKay Type I distributions, from we which we conjecture some tight two-sided inequalities for their medians.    Numerical experiments support these conjectures and also lead us to a conjectured tight lower bound for the median of the GH distribution.




\end{abstract}

\noindent{{\bf{Keywords:}}} Generalized hyperbolic distribution; variance-gamma distribution; McKay Type I distribution; mode; median; inequality; modified Bessel function

\noindent{{{\bf{AMS 2010 Subject Classification:}}} Primary 60E05; 62E15

\section{Introduction}
Consider the generalized hyperbolic (GH) distribution with probability density function (PDF)
\begin{equation}\label{ghpdf}p_{\mathrm{GH}}(x)=\frac{(\gamma/\delta)^\lambda}{\sqrt{2\pi}K_\lambda(\delta\gamma)}\mathrm{e}^{\beta(x-\mu)}\frac{K_{\lambda-1/2}(\alpha\sqrt{\delta^2+(x-\mu)^2})}{(\sqrt{\delta^2+(x-\mu)^2}/\alpha)^{1/2-\lambda}}, \quad x\in\mathbb{R},
\end{equation}
where $\gamma=\sqrt{\alpha^2-\beta^2}$ and $K_\lambda$ is a modified Bessel function of the second kind (see Appendix \ref{appendix} for a definition and basic properties used in this paper). If a random variable $X$ has density (\ref{ghpdf}), we write $X\sim GH(\lambda,\alpha,\beta,\delta,\mu)$.  The parameter domain is given by
\begin{align*}&\delta\geq0, \quad \gamma>0, \quad \lambda>0, \\
&\delta>0, \quad \gamma>0, \quad \lambda=0, \\
&\delta>0, \quad \gamma\geq0, \quad \lambda<0,
\end{align*}
and in all cases $\mu\in\mathbb{R}$.  If $\delta=0$ or $\gamma=0$, then the density (\ref{ghpdf}) is defined as the limit obtained by using (\ref{Ktend0}).  In particular, taking $\delta\downarrow0$ in (\ref{ghpdf}) yields
\begin{equation} \label{seven} p_{\mathrm{VG}_2}(x) = \frac{\gamma^{2\lambda}}{\sqrt{\pi} \Gamma(\lambda)}\left(\frac{|x-\mu|}{2\alpha}\right)^{\lambda-1/2} e^{\beta (x-\mu)}K_{\lambda-1/2}(\alpha|x-\mu|), \quad x\in \mathbb{R},
\end{equation}
which is the PDF of the variance-gamma (VG) distribution.  If a random variable $X$ has density (\ref{seven}), we write $X\sim\mathrm{VG}_2(\lambda,\alpha,\beta,\mu)$, and another parametrisation of the distribution that will be used in this paper is given in equation (\ref{vgdef}).  For the GH and VG distributions $\mu$ is the location parameter, and in this paper we shall set it to 0 to simplify the exposition; results for the general case follow by a
simple linear transformation, since $GH(\lambda,\alpha,\beta,\delta,\mu)=_d \mu+GH(\lambda,\alpha,\beta,\delta,0)$.

The GH distribution was introduced by Barndorff-Nielsen \cite{barndorff, barndorff2}, who studied it in the context of modelling dune movements.  The GH distributions are often an excellent fit to financial data, in part because of their semi-heavy tails, and are widely used in financial modelling; see for example, \cite{bibby, eberlein0, eberlein1, eberlein2, mcneil, r99}.   The GH distribution also posses an attractive distributional theory, with special and limiting cases including, amongst others, the VG, hyperbolic, normal-inverse Gaussian and Student's $t$-distributions \cite{eberlein}.
The VG distribution, itself, is also widely used in financial modelling; see, for example, \cite{mcc98, madan, s04}.


Accounts of the distributional theory of GH distributions and their most basic distributional properties, such as the characteristic function, mean, variance, skewness, kurtosis and L\'{e}vy-Khintchine representation can be found in \cite{bibby,hammerstein}, and even some more exotic properties such as Stein characterisations \cite{gaunt gh} and Stein-type lemmas \cite{vy17} have been established.  However, the presence of the modified Bessel function of the second kind in the density (\ref{ghpdf}) can complicate the analysis for certain distributional properties, such as the calculation of moments and absolute moments of general order \cite{bs05,scott}.  This is also the case for the mode of the distribution.  On applying the differentiation formula (\ref{ddbk}) to (\ref{ghpdf}) it follows that the mode $M$ is the unique solution of the equation
\begin{equation}\label{xstar}\frac{x}{\sqrt{\delta^2+x^2}}\cdot\frac{K_{\lambda-3/2}\big(\alpha\sqrt{\delta^2+x^2}\big)}{K_{\lambda-1/2}\big(\alpha\sqrt{\delta^2+x^2}\big)}=\frac{\beta}{\alpha}.
\end{equation}
That $M$ is the unique solution follows from the fact that the entire class of GH distributions is unimodal \cite{yu}.

If $\beta=0$ (in which case the distribution is symmetric about the origin), it can be readily seen from (\ref{xstar}) that $M=0$.  For some other special parameter values, one can also obtain closed-form expressions for the mode; see Section \ref{sec2.4}.  However, in general, one cannot obtain a closed-form solution to (\ref{xstar}).  Motivated by this, in this paper we  exploit the extensive literature on inequalities for ratios of modified Bessel functions of the second kind to obtain a simple but accurate two-sided inequality for the mode of the GH distribution (Theorem \ref{thmghd}), which is the first such result in the literature.  As a consequence, we obtain a bound for the difference between the mean and the mode of the distribution (Corollary \ref{corghd}).  The inequalities take a particularly neat form for the VG subclass, and we collect them in Corollary \ref{thmvg}.

The VG distribution is also sometimes referred to as the McKay Type II distribution; the McKay Type I distribution \cite{mckay}, which is used in performance-analysis problems in wireless communication systems \cite{ha04}, has a PDF with a similar functional form with the modified Bessel function $K_\nu$ replaced by the modified Bessel function of the first kind $I_\nu$. For $b>0$, $c>1$, $m>-1/2$ the PDF is
\begin{equation}\label{seven7} p_{\text{McKay I}}(x)=\frac{\sqrt{\pi}(c^2-1)^{m+1/2}}{2^mb^{m+1}\Gamma(m+\frac{1}{2})}x^m\mathrm{e}^{-cx/b}I_m\Big(\frac{x}{b}\Big), \quad x>0.
\end{equation}
Given the similarity between the PDFs of the McKay Type I and VG distributions, it is natural to extend our study to the mode of this distribution, which has yet to be studied in the literature.  Through the differentiation formula (\ref{ddbi}) we see that the mode  $M$ is the unique solution of the equation
\begin{equation}\label{ystar}\frac{I_m(x/b)}{I_{m-1}(x/b)}=\frac{1}{c}.
\end{equation}
In Theorem \ref{thmmckay}, we show that the McKay Type I distribution is unimodal, and we use bounds from the literature on inequalities for ratios of modified Bessel functions of the first kind to obtain tight two-sided inequalities for the mode.

It should be noted that, for given fixed parameter values, equations (\ref{xstar}) and (\ref{ystar}) can be solved via numerical methods such as the Newton-Raphson method, which would lead to estimates for the mode of the GH, VG and McKay Type I distributions that are more accurate than the bounds obtained in this paper.  The purpose of our paper, however, is to give accurate analytic bounds than hold for a wide range of parameter values.  Our bounds can also be used to produce informed initial guesses for numerical methods.


Having obtained bounds for the mode of the GH, VG and McKay Type I distributions, it is natural to seek analogous results for the median, which is the focus of Section \ref{sec3}.  This turns out to be a harder problem.
We contend ourselves with conjecturing some monotonicity results for the medians of the VG and McKay Type I distributions.  These results are motivated through the representation of the VG and McKay Type I distributions as the difference and sum, respectively, of two independent gamma random variables.  As a consequence, we are able to conjecture a tight two-sided inequality for the median of the McKay Type I distribution in terms of medians of gamma distributions, as well as an upper bound for the VG distribution in terms of the median of the gamma distribution.  Following the conjecture of \cite{cr86}, the problem of bounding the median of the gamma distribution has received considerable attention in the literature, and we are able to draw on this literature to conjecture tight bounds for the VG and McKay distribution.  In addition, we conjecture tight lower bounds for the median of the GH and VG distributions. We support these conjectures with numerical results, and also detail an approach involving Schur convexity and stochastic ordering that could ultimately lead to a proof of the conjectured monotonicity results for sums and differences of independent gamma random variables.  We hope our study will inspire future work towards verifying these conjectures.


\section{Inequalities for the mode of the generalized hyperbolic and related distributions}\label{sec2}

The following bounds for the ratio $K_{\nu-1}(x)/K_{\nu}(x)$ will be used in the sequel. For $x>0$,
\begin{equation}\label{seg1}\frac{x}{\nu+\sqrt{\nu^2+x^2}}<\frac{K_{\nu-1}(x)}{K_\nu(x)}<\frac{x}{\nu-1+\sqrt{(\nu-1)^2+x^2}}, \quad \nu\in\mathbb{R},
\end{equation}
and
\begin{equation}\label{seg2}\frac{K_{\nu-1}(x)}{K_\nu(x)}>\frac{x}{\nu-1/2+\sqrt{(\nu-1/2)^2+x^2}}, \quad \nu>1/2.
\end{equation}
We have equality in (\ref{seg2}) for $\nu=1/2$ and the inequality is reversed for  $\nu<1/2$.  The lower bound in (\ref{seg1}) was proved by \cite{ln10} and, as noted in Remark 2 of \cite{rs16}, the upper bound can in fact be derived by a very short calculation that uses the lower bound and the fact that $K_{\nu}(x)=K_{-\nu}(x)$ (see (\ref{parity})).  Inequality (\ref{seg2}) was established by \cite{segura} and the fact that the inequality is reversed for $\nu<1/2$ follows from exactly the same argument as used in Remark 2 of \cite{rs16}.

An application of an iterative scheme given in Section 3.2 of \cite{segura} yields other bounds for the ratio $K_{\nu-1}(x)/K_{\nu}(x)$ that will be used in this paper.  Consider the relation
\begin{equation}\label{relk}\frac{K_{\nu-1}(x)}{K_\nu(x)}=\bigg(\frac{2(\nu-1)}{x}+\frac{K_{\nu-2}(x)}{K_{\nu-1}(x)}\bigg)^{-1}
\end{equation}
(see \cite[equation (29)]{segura}).  Then applying inequality (\ref{seg2}) to the right-hand side of (\ref{relk}) yields the following inequality.  For $x>0$,
\begin{equation}\label{seg3}\frac{K_{\nu-1}(x)}{K_\nu(x)}<\frac{x}{\nu-1/2+\sqrt{(\nu-3/2)^2+x^2}}, \quad \nu>3/2.
\end{equation}
We have equality for $\nu=3/2$ and the inequality is reversed for  $\nu<3/2$.  In its range of validity (\ref{seg3}) outperforms the upper bound in (\ref{seg1}).  Also, the reversed inequality outperforms the lower bound (\ref{seg2}) for $1<\nu<3/2$, but the lower bound (\ref{seg2}) is more accurate for $1/2<\nu<1$.

We will also need the following two-sided inequality for a ratio of modified Bessel functions of the first kind, due to \cite{segura} (see also \cite{amos74,ln10,rs16}), which states that, for $x>0$,
\begin{equation}\label{sqrtbb}\frac{x}{\nu-1/2+\sqrt{(\nu+1/2)^2+x^2}}<\frac{I_{\nu}(x)}{I_{\nu-1}(x)}<\frac{x}{\nu-1/2+\sqrt{(\nu-1/2)^2+x^2}},
\end{equation}
where the lower bound holds for $\nu\geq0$ and the upper bound is valid for $\nu\geq1/2$.  We will also need the following bound of \cite{rs16}
\begin{equation}\label{segi2}\frac{I_{\nu}(x)}{I_{\nu-1}(x)}<\frac{x}{\nu-1+\sqrt{(\nu+1)^2+x^2}}, \quad x>0,\: \nu\geq0.
\end{equation}
 which improves on the upper bound in (\ref{sqrtbb}) for $0<x<2\sqrt{\nu(2\nu+1)}$.


\subsection{The mode of the generalized hyperbolic distribution}

In what follows, we fix $\beta>0$; all inequalities are reversed for $\beta<0$ because, for  $X_1\sim GH(\lambda,\alpha,\beta,\delta,0)$ and $X_2\sim GH(\lambda,\alpha,-\beta,\delta,0)$, $X_1$ is equal in law to $-X_2$.

\begin{theorem}\label{thmghd}Let $X\sim GH(\lambda,\alpha,\beta,\delta,0)$ and let $M$ denote its mode. Fix $\beta>0$.

\vspace{2mm}

\noindent (i) For $\delta>0$ and $\lambda\in\mathbb{R}$,
\begin{align}\label{mode1}\frac{\beta}{\gamma^2}\Big[\lambda-3/2+\sqrt{(\lambda-3/2)^2+\delta^2\gamma^2}\Big]< M<\frac{\beta}{\gamma^2}\Big[\lambda-1/2+\sqrt{(\lambda-1/2)^2+\delta^2\gamma^2}\Big].
\end{align}
The lower bound also holds for $\delta=0$, $\lambda>3/2$, and the upper bound also holds for $\delta=0$, $\lambda>1/2$.

For a restricted range of parameter values, the upper and lower bounds of (\ref{mode1}) can be improved as follows. Let $\delta>0$. Then, for $\lambda>1$,
\begin{equation}\label{mode2} M<\frac{\beta}{\gamma^2}\Big[\lambda-1+\sqrt{(\lambda-1)^2+\delta^2\gamma^2}\Big].
\end{equation}
We have equality in (\ref{mode2}) if $\lambda=1$, and the inequality is reversed for $\lambda<1$.  The upper bound (\ref{mode2}) is also valid for $\delta=0$, $\lambda>1$ and we have equality for $\delta=0$, $\lambda=1$ (that is $M=0$).  Also, for $\delta>0$, $\lambda>2$,
\begin{equation}\label{mode29} M>\frac{\beta}{\gamma^2}\bigg[\lambda-1+\sqrt{(\lambda-1)^2+\gamma^2\bigg(\delta^2+\frac{3-2\lambda}{\alpha^2}\bigg)}\bigg].
\end{equation}
We have equality in (\ref{mode29}) if $\lambda=2$, and the inequality is reversed for $\lambda<2$.  The lower bound (\ref{mode29}) is also valid for $\delta=0$, $\lambda>2$, we have equality for $\delta=0$, $\lambda=2$, and the reverse inequality also holds for $\delta=0$, $1<\lambda<2$.
\vspace{2mm}

\noindent (ii) The two-sided inequality (\ref{mode1}) is tight in the limits $\lambda\rightarrow\infty$ and $\delta\rightarrow\infty$.  Moreover, letting $\sigma=1/\gamma$ and $\theta=\beta/\gamma^2$, for fixed $\theta$, we have that in the limit $\sigma\rightarrow\infty$,
\begin{align}M&\sim \begin{cases}\displaystyle \theta(2\lambda-3), & \quad \lambda>3/2, \:\delta\geq0, \\[2pt]
\displaystyle\frac{\theta}{\log \sigma}, & \quad \lambda=3/2,\:\delta\geq0, \\[12pt]
\displaystyle \frac{2^{2\lambda-2}\Gamma(\lambda-\frac{1}{2})}{\Gamma(\frac{3}{2}-\lambda)}\frac{\theta\delta^{3-2\lambda}}{\sigma^{3-2\lambda}}, & \quad 1/2<\lambda<3/2,\:\delta>0, \\[12pt]
\label{asym4}\displaystyle \frac{\theta\delta^2}{(1-2\lambda)\sigma^2}, & \quad \lambda<1/2, \:\delta>0. \end{cases}
\end{align}
Additionally, if $\delta=0$, then $M=0$ if $0<\lambda\leq1$, and if $\lambda>1$ in the limit $\sigma\downarrow0$,  we have
\begin{equation}\label{asym5}M\sim 2\theta(\lambda-1), \quad \lambda>1.
\end{equation}
Now let $b_a(\lambda,\alpha,\beta,\delta)=\frac{\beta}{\gamma^2}[\lambda -a+\sqrt{(\lambda -a)^2+\delta^2\gamma^2}]$. It follows from (\ref{asym4}) (cases $\lambda>3/2$ and $\lambda<1/2$) and (\ref{asym5}), respectively, that lower bound of (\ref{mode1}), the upper bound of (\ref{mode1}) and inequality (\ref{mode2}), with values $a=3/2$, $a=1/2$ and $a=1$, respectively, are best possible, in their range of validity, amongst all bounds of the form $b_a(\lambda,\alpha,\beta,\delta)$ with universal constants $a\in\mathbb{R}$ that do not involve the parameters $\lambda$, $\alpha$, $\beta$ and $\delta$.

\end{theorem}

\begin{proof} (i) Let us derive the upper bound in (\ref{mode1}).  We have that $M$ is the unique solution to (\ref{xstar}).  Let $\lambda\in\mathbb{R}$.  By applying the lower bound of (\ref{seg1}) to (\ref{xstar}) we obtain the following inequality for $x$:
\begin{equation}\label{inti}\frac{x}{\sqrt{\delta^2+x^2}}\cdot\frac{\alpha \sqrt{\delta^2+x^2}}{\lambda-1/2+\sqrt{(\lambda-1/2)^2+\alpha^2(\delta^2+x^2)}}<\frac{\beta}{\alpha}.
\end{equation}
The assumption $\delta>0$ ensures that, for all $\lambda\in\mathbb{R}$, the left-hand side of (\ref{inti}) is strictly less than $\beta/\alpha$ for some $x>0$. If $\delta=0$ and we restrict $\lambda>1/2$, then it can be seen that the left-hand side of (\ref{inti}) is strictly less than $\beta/\alpha$ for some $x>0$.  However, if $\delta=0$ and $\lambda\leq1/2$, then the left-hand side is strictly greater than $1>\beta/\alpha$, meaning that there is no solution $x>0$ to the inequality. Simplifying and rearranging (\ref{inti}), in the cases $\delta>0,\lambda\in\mathbb{R}$ or $\delta=0,\lambda>1/2$, yields a quadratic inequality which has solution
\begin{equation*}x<\frac{\beta}{\gamma^2}\Big[\lambda-1/2+\sqrt{(\lambda-1/2)^2+\delta^2\gamma^2}\Big],
\end{equation*}
from which we deduce the upper bound (\ref{mode1}) for $M$.  For the other inequalities in the statement of the theorem, one also needs to carefully consider for which values of $\lambda$ and $\delta$ the inequalities are valid.  For reasons of brevity and ease of reading, we omit these details in the remainder of the proof.  To obtain the lower bound of (\ref{mode1}), we apply the upper bound of (\ref{seg1}) to (\ref{xstar}) and proceed in the same manner.  For $\lambda>1$, we can apply the refined inequality (\ref{seg2}) to (\ref{xstar}), which yields inequality (\ref{mode2}); the cases $\lambda=1$ and $\lambda<1$ are dealt with similarly.  In the same manner, for $\lambda>2$ (and $\lambda=2$, $\lambda<2$) we obtain inequality (\ref{mode29}) by this time applying inequality (\ref{seg3}) to (\ref{xstar}).

\vspace{2mm}

\noindent (ii) It is clear that the two-sided inequality (\ref{mode1}) is sharp as $\lambda\rightarrow\infty$ and $\delta\rightarrow\infty$.  The more involved analysis involves the limits as $\sigma\rightarrow\infty$ and $\sigma\downarrow0$. We establish the four cases of (\ref{asym4}) separately. With $\sigma=1/\gamma$ and $\theta=\beta/\gamma^2$ equation (\ref{xstar}) reads
\begin{equation}\label{xstar2}f_{\lambda,\sigma}(x):=\frac{\sqrt{\theta^2+\sigma^2}}{\theta}\cdot\frac{x}{\sqrt{\delta^2+x^2}}\cdot\frac{K_{\lambda-3/2}\big(\frac{\sqrt{\theta^2+\sigma^2}}{\sigma^2}\sqrt{\delta^2+x^2}\big)}{K_{\lambda-1/2}\big(\frac{\sqrt{\theta^2+\sigma^2}}{\sigma^2}\sqrt{\delta^2+x^2}\big)}=1.
\end{equation}
Applying (\ref{Ktend0}) to (\ref{xstar2}) gives that, as $\sigma\rightarrow\infty$,
\begin{align}f_{\lambda,\sigma}(x)\sim \begin{cases}\displaystyle \frac{x}{\theta(2\lambda-3)}, & \quad \lambda>3/2,\:\delta\geq0, \\[12pt]
\displaystyle\frac{x\log\sigma}{\theta}, & \quad \lambda=3/2,\:\delta\geq0, \\[12pt]
\displaystyle \frac{\Gamma(\frac{3}{2}-\lambda)}{\Gamma(\lambda-\frac{1}{2})}\frac{2^{2-2\lambda}\sigma^{3-2\lambda}}{\theta }x(\delta^2+x^2)^{\lambda-3/2}, & \quad 1/2<\lambda<3/2,\:\delta>0, \\[12pt]
\label{asym55}\displaystyle \frac{(1-2\lambda)\sigma^2 x}{\theta(x^2+\delta^2)}, & \quad \lambda<1/2, \:\delta>0. \end{cases}
\end{align}
Here we used that $\sqrt{\theta^2+\sigma^2}\sim \sigma$, as $\sigma\rightarrow\infty$, and the standard formula $\Gamma(u+1)=u\Gamma(u)$, as well as that $x\ll\sigma$, which can be seen from inequality (\ref{mode1}) which tells us that $x=O(1)$, as $\sigma\rightarrow\infty$. The limiting forms for the cases $\lambda>3/2$ and $\lambda=3/2$ of (\ref{asym4}) immediately follow from combining (\ref{xstar2}) and (\ref{asym55}).  Consider now the case $1/2<\lambda<3/2$. Suppose that the leading term in the asymptotic expansion of $x$ is of the form $x\sim a_0\sigma^m+o(\sigma^m)$, as $\sigma\rightarrow\infty$, where $a_0$ and $m$ are constants to be determined.  Then combining (\ref{xstar2}) and (\ref{asym55}) gives that
\begin{align*} \frac{\Gamma(\frac{3}{2}-\lambda)}{\Gamma(\lambda-\frac{1}{2})}\frac{2^{2-2\lambda}}{\theta }a_0\sigma^{3-2\lambda+m}\big(1+o(1)\big)\big(\delta^2+a_0^2\sigma^{2m}(1+o(1))\big)^{\lambda-3/2}=1.
\end{align*}
It follows that we must take $m=2\lambda-3$ and therefore, as $1/2<\lambda<3/2$, we have $x\sim a_0\sigma^{2\lambda-3}$, as $\sigma\rightarrow\infty$, so that $x\ll\delta$.  Therefore $a_0$ satisfies
\begin{align*}\frac{\Gamma(\frac{3}{2}-\lambda)}{\Gamma(\lambda-\frac{1}{2})}\frac{2^{2-2\lambda}}{\theta }a_0\delta^{2\lambda-3}=1,
\end{align*}
whence on rearranging for $a_0$ we deduce the limiting form (\ref{asym4}) in the case $1/2<\lambda<3/2$. For the case $\lambda<1/2$ a similar analysis leads to the desired limiting form $M\sim\frac{\theta\delta^2}{(1-2\lambda)\sigma^2}$.

Now let $\delta=0$. If $\lambda=1$, the density (\ref{seven}) becomes $p(x)=N_{\theta,\sigma}\mathrm{e}^{(x\theta  -|x|\sqrt{\theta^2+\sigma^2})/\sigma^2}$, $x\in\mathbb{R}$, where $N_{\theta,\sigma}$ is the normalising constant, and therefore $M=0$ when $\lambda=1$.  For $0<\lambda<1$, we see from (\ref{Ktend0}) that the density (\ref{seven}) blows up as $x\rightarrow0$ (since $\delta=0$), and so the unique mode is $0$.
Now we consider the case $\lambda>1$ and establish the limiting form (\ref{asym5}).  First we let $\delta=0$ in (\ref{xstar}) to obtain
\begin{equation}\label{xstar3}f_{\lambda,0}(x):=\frac{\sqrt{\theta^2+\sigma^2}}{\theta}\cdot\frac{K_{\lambda-3/2}\big(\frac{\sqrt{\theta^2+\sigma^2}}{\sigma^2}x\big)}{K_{\lambda-1/2}\big(\frac{\sqrt{\theta^2+\sigma^2}}{\sigma^2}x\big)}=1.
\end{equation}
Applying (\ref{Ktendinfinity}) to (\ref{xstar3}) gives that, as $\sigma\downarrow0$,
\begin{align} f_{\lambda,0}(x)&= \bigg(1+\frac{\sigma^2}{2\theta^2}\bigg)\frac{\displaystyle1+(4(\lambda-3/2)^2-1)\cdot\frac{\sigma^2}{8x}}{\displaystyle1+(4(\lambda-1/2)^2-1)\cdot\frac{\sigma^2}{8x}}+O(\sigma^4)\nonumber \\
\label{order4}&=1+\frac{(x+2-2\lambda)\sigma^2}{2x}+O(\sigma^4).
\end{align}
On combining (\ref{xstar3}) and (\ref{order4}) we deduce that in the case $\delta=0$, we have $M\sim 2\theta(\lambda-1)$, as $\sigma\downarrow0$, as required.

Finally, we prove that the constants $a=3/2$, $a=1/2$ and $a=1$ in the double inequality (\ref{mode1}) and inequality (\ref{mode2}) are best possible.  Fix $\theta$.  Then
\begin{align*}b_{\frac{3}{2}}(\lambda,\alpha,\beta,\delta)&\sim \theta(2\lambda-3), \quad \sigma\rightarrow\infty,\: \lambda>3/2, \: \delta\geq0, \\
b_{\frac{1}{2}}(\lambda,\alpha,\beta,\delta)&\sim \frac{\theta\delta^2}{(1-2\lambda)\sigma^2}, \quad \sigma\rightarrow\infty,\: \lambda<1/2,\: \delta>0, \\
b_1(\lambda,\alpha,\beta,\delta)&\sim 2\theta(\lambda-1), \quad \sigma\downarrow0,\: \lambda>1,\: \delta=0,
\end{align*}
which are in agreement with the limiting forms (\ref{asym4}) (cases $\lambda>3/2$ and $\lambda<1/2$) and (\ref{asym5}) for the mode of the GH distribution, from which it follows that the constants $a=3/2$, $a=1/2$ and $a=1$ are best possible.
\end{proof}

\begin{remark}As we shall see later, the re-parametrisation $\sigma=1/\gamma$ and $\theta=\beta/\gamma^2$ is also convenient when working with the VG distribution.
\end{remark}

It is of interest to estimate the discrepancy between the mean and mode of the GH distribution.  Let $X\sim GH(\lambda,\alpha,\beta,\delta,0)$.  Then the following exact formula is available \cite{bb80}:
\begin{equation}\label{expfor}\mathbb{E}[X]=\frac{\delta\beta K_{\lambda+1}(\delta\gamma)}{\gamma K_\lambda(\delta\gamma)}.
\end{equation}

\begin{proposition}\label{pmean}Let $X\sim GH(\lambda,\alpha,\beta,\delta,0)$.  Fix $\beta>0$.  Then, for $\lambda\in\mathbb{R}$,
\begin{equation}\label{mean1} \frac{\beta}{\gamma^2}\Big[\lambda+\sqrt{\lambda^2+\delta^2\gamma^2}\Big]<\mathbb{E}[X]<\frac{\beta}{\gamma^2}\Big[\lambda+1+\sqrt{(\lambda+1)^2+\delta^2\gamma^2}\Big].
\end{equation}
For a restricted range of parameter values, the upper and lower bounds of (\ref{mean1}) can be improved as follows. For $\lambda>-1/2$, we have
\begin{equation}\label{mean2} \mathbb{E}[X]<\frac{\beta}{\gamma^2}\Big[\lambda+1/2+\sqrt{(\lambda+1/2)^2+\delta^2\gamma^2}\Big].
\end{equation}
 We have equality in (\ref{mean2}) if $\lambda=-1/2$, and the inequality is reversed for $\lambda<-1/2$.  Also, for $\lambda>1/2$,
\begin{equation}\label{mean29} \mathbb{E}[X]>\frac{\beta}{\gamma^2}\Big[\lambda+1/2+\sqrt{(\lambda-1/2)^2+\delta^2\gamma^2}\Big].
\end{equation}
 We have equality in (\ref{mean29}) if $\lambda=1/2$, and the inequality is reversed for $\lambda<1/2$.
\end{proposition}

\begin{proof}Inequality (\ref{mean1}) follows from applying inequality (\ref{seg1}) to (\ref{expfor}), whilst inequalities (\ref{mean2}) and (\ref{mean2}) are obtained by instead applying inequalities (\ref{seg2}) and (\ref{mean2}), respectively.
\end{proof}

\begin{corollary}\label{corghd}Let $X\sim GH(\lambda,\alpha,\beta,\delta,0)$.  Fix $\beta>0$.  Then, for $\lambda\in\mathbb{R}$,
\begin{equation}\label{modemean}M<\mathbb{E}[X],
\end{equation}
and we have the following refined inequalities. For $\lambda\in\mathbb{R}$,
\begin{align}&\frac{\beta}{\gamma^2}\Big[1/2+\sqrt{\lambda^2+\delta^2\gamma^2}-\sqrt{(\lambda-1/2)^2+\delta^2\gamma^2}\Big]<\mathbb{E}[X]-M<\nonumber \\
\label{modemean0}&\quad<\frac{\beta}{\gamma^2}\Big[5/2+\sqrt{(\lambda+1)^2+\delta^2\gamma^2}-\sqrt{(\lambda-3/2)^2+\delta^2\gamma^2}\Big].
\end{align}
For a restricted range of parameter values, the upper and lower bounds of (\ref{modemean0}) can be improved as follows. For $\lambda\geq1$,
\begin{equation}\label{modemean1}\mathbb{E}[X]-M>\frac{\beta}{\gamma^2}\Big[3/2+\sqrt{(\lambda-1/2)^2+\delta^2\gamma^2}-\sqrt{(\lambda-1)^2+\delta^2\gamma^2}\Big],
\end{equation}
and
\begin{align*}\label{modemean2}\mathbb{E}[X]-M&<\frac{\beta}{\gamma^2}\Big[2+\sqrt{(\lambda+1/2)^2+\delta^2\gamma^2}-\sqrt{(\lambda-3/2)^2+\delta^2\gamma^2}\Big], \quad \lambda\geq-1/2, \\
\mathbb{E}[X]-M&<\frac{\beta}{\gamma^2}\bigg[3/2+\sqrt{(\lambda+1/2)^2+\delta^2\gamma^2}-\sqrt{(\lambda-1)^2+\gamma^2\bigg(\delta^2+\frac{3-2\lambda}{\alpha^2}\bigg)}\bigg], \quad \lambda\geq2.
\end{align*}
\end{corollary}

\begin{proof} Combining the two-sided inequalities (\ref{mode1}) and (\ref{mean1}) yields the two-sided inequality (\ref{modemean0}). We obtain inequality (\ref{modemean1}) from inequality (\ref{mode2}) and the lower bound (\ref{mean29}). We obtain the two upper bounds for $\mathbb{E}[X]-M$ by combining the lower bound of (\ref{mode1}) (for $\lambda\geq-1/2$) and inequality (\ref{mode29}) (for $\lambda\geq2$) with the upper bound (\ref{mean2}). Finally, inequality (\ref{modemean}) can be deduced from the lower bound of (\ref{modemean0}) because, for $a>0$ and $b,u\in\mathbb{R}$, the elementary inequality $a+\sqrt{u^2+b^2}-\sqrt{(u-a)^2+b^2}>0$ holds.
\end{proof}

\begin{remark}Several other inequalities for $\mathbb{E}[X]-M$ can be obtained by combining other inequalities from Theorem \ref{thmghd} and Proposition \ref{pmean}.  We omit these for reasons of brevity.
\end{remark}



\subsection{The mode of the variance-gamma distribution}

Recall that $\lim_{\delta\downarrow0}GH(\lambda,\alpha,\beta,\delta,\mu)=\mathrm{VG}_2(\lambda,\alpha,\beta,\mu)$.  Therefore bounds for mode of the $\mathrm{VG}_2(\lambda,\alpha,\beta,0)$ distribution follow immediately from setting $\delta=0$ in the bounds of Theorem \ref{thmghd}.  The following alternative parametrisation of the VG is, however, quite convenient, so we present results for this parametrisation as well.

The VG distribution with parameters $r > 0$, $\theta \in \mathbb{R}$, $\sigma >0$, $\mu \in \mathbb{R}$ has PDF
\begin{equation}\label{vgdef}p_{\mathrm{VG}}(x) = \frac{1}{\sigma\sqrt{\pi} \Gamma(\frac{r}{2})} \mathrm{e}^{\theta (x-\mu)/\sigma^2} \bigg(\frac{|x-\mu|}{2\sqrt{\theta^2 +  \sigma^2}}\bigg)^{\frac{r-1}{2}} K_{\frac{r-1}{2}}\bigg(\frac{\sqrt{\theta^2 + \sigma^2}}{\sigma^2} |x-\mu| \bigg), \quad x\in\mathbb{R},
\end{equation}
If a random variable $X$ has density (\ref{vgdef}) then we write $X\sim \mathrm{VG}(r,\theta,\sigma,\mu)$.  This parametrisation was given in \cite{gaunt vg}, and is related to the one given in (\ref{seven}) by
\begin{equation*}\label{para}r=2\lambda, \quad \theta=\beta/\gamma^2, \quad \sigma=1/\gamma.
\end{equation*}
 It is similar to the parametrisation given by \cite{finlay} and another parametrisation that differs to (\ref{seven}) is given in the book \cite{kkp01}, in which the name generalized Laplace distribution is used.  Similarly, to how we proceeded for the GH distribution, we fix $\theta>0$ in the sequel.

\begin{corollary}\label{thmvg}Let $X\sim\mathrm{VG}(r,\theta,\sigma,0)$ with $\theta>0$ and $\sigma>0$.   If $0<r\leq2$, then $M=0$.  If $r>2$, then $M>0$ and the following two-sided inequality holds:
\begin{equation}\label{mode1vg} \theta(r-3)<M<\theta(r-2),
\end{equation}
and we also have that
\begin{equation}\label{vgmk}2\theta<\mathbb{E}[X]-M<3\theta.
\end{equation}
For $r>4$,
\begin{equation}\label{r4r4}M>\frac{\theta}{2}\bigg[r-2+\sqrt{\frac{\theta^2(r-2)^2+\sigma^2(r-4)^2}{\theta^2+\sigma^2}}\bigg],
\end{equation}
which improves on the lower bound in (\ref{mode1vg}) in the range of validity $r>4$.  We have equality in (\ref{r4r4}) if $r=4$ and the inequality is reversed if $2<r<4$.  If $3<r<4$ then the reversed inequality (\ref{r4r4}) improves on the upper bound in (\ref{mode1vg}), but the reverse is true for $2<r<3$.

The lower and upper bounds in (\ref{mode1vg}) are best possible amongst all bounds of the form $\theta(r-a)$, where $a\in\mathbb{R}$ does not involve the model parameters.  Also, the lower and upper bounds in (\ref{vgmk}) are best possible amongst all bounds of the form $b\theta$, where $b\in\mathbb{R}$ does not involve the model parameters.

\end{corollary}

\begin{proof} That $M=0$ for $0\leq r\leq2$ and inequality (\ref{mode1vg}) follow from Theorem \ref{thmghd} (with $\delta=0$) after making the change of parameters $r=2\lambda$, $\theta=\beta/\gamma^2$, $\sigma=1/\gamma$.  It should be noted that the lower bound in (\ref{mode1vg}) only follows from inequality (\ref{mode1}) of Theorem \ref{thmghd} for $r>3$, but it is also valid for $r>2$ because $M>0$ for $r>2$.  The latter assertion follows because the mode of the $\mathrm{VG}(r,\theta,\sigma,0)$ distribution is given by the unique non-negative solution to (\ref{xstar3}) (with $\lambda=2r$).  By Corollary 2.3 of \cite{gaunt ineq1} there is exactly one solution to $C K_{\nu-1}(y)/K_\nu(y)=1$, $C>1$, $\nu>1/2$ in the region $y>0$ and therefore we conclude that there is exactly one strictly positive solution to (\ref{xstar3}) for $r>2$, and so $M>0$ for $r>2$.

 The mean of the VG distribution takes a simple form: $\mathbb{E}[X]=r\theta$ for $X\sim\mathrm{VG}(r,\theta,\sigma,0)$  \cite{eberlein}.  Combining this formula with inequality (\ref{mode1vg}) yields (\ref{vgmk}).  Inequality (\ref{r4r4}) follows from inequality (\ref{mode29}) with $\delta=0$ and a calculation to verify that
\begin{align*}\frac{\beta}{\gamma^2}\bigg[\lambda-1+\sqrt{(\lambda-1)^2+\frac{\gamma^2(3-2\lambda)}{\alpha^2}}\bigg]=\frac{\theta}{2}\bigg[r-2+\sqrt{\frac{\theta^2(r-2)^2+\sigma^2(r-4)^2}{\theta^2+\sigma^2}}\bigg].
\end{align*}
Elementary calculations can be used to confirm the assertions regarding the relations between inequality (\ref{r4r4}) and the bounds in (\ref{mode1vg}).  The assertion that the lower and upper bounds (\ref{mode1vg}) are best possible amongst all bounds of the form $\theta(r-a)$, where $a\in\mathbb{R}$ does not involve the model parameters, follows from Theorem \ref{thmghd} (with $\delta=0$).  The assertion regarding the optimality of the double inequality (\ref{vgmk}) then follows because an exact formula was used for the mean $\mathbb{E}[X]$.
\end{proof}

\begin{remark}Let $B_{r,\theta,\sigma}$ denote the bound for $M$ in (\ref{r4r4}).  Then, for fixed $r>4$ and $\theta>0$, the function $\sigma\mapsto B_{r,\theta,\sigma}$ is a strictly decreasing on $(0,\infty)$ and
\[\lim_{\sigma\downarrow0}B_{r,\theta,\sigma}=\theta(r-2), \quad \lim_{\sigma\rightarrow\infty}B_{r,\theta,\sigma}=\theta(r-3).\]
\end{remark}






Let $(X, Y)$ be a bivariate normal random vector with a zero-mean vector, variances $(\sigma_X^2, \sigma_Y^2)$, and correlation coefficient $\rho$. Denote the product $Z=XY$ and let $Z_1,Z_2,\ldots,Z_n$ be independent copies of $Z$. Then the exact distribution of the sample mean $\overline{Z}=\frac{1}{n}(Z_1+Z_2\cdots+Z_n)$ was obtained by \cite{np16}:
\begin{equation*}p_{\overline{Z}}(x)=\frac{n^{(n+1)/2}2^{(1-n)/2}|x|^{(n-1)/2}}{(\sigma_X\sigma_Y)^{(n+1)/2}\sqrt{\pi(1-\rho^2)}\Gamma\big(\frac{n}{2}\big)}\exp\bigg(\frac{\rho n x}{\sigma_X\sigma_Y(1-\rho^2)} \bigg)K_{\frac{n-1}{2}}\bigg(\frac{n |x|}{\sigma_X\sigma_Y(1-\rho^2)}\bigg),
\end{equation*}
$x\in\mathbb{R}$, and it was noted by \cite{gaunt prod} that $\overline{Z}\sim\mathrm{VG}(n,\frac{1}{n}\rho\sigma_X\sigma_Y,\frac{1}{n}\sigma_X\sigma_Y\sqrt{1-\rho^2},0)$.  The following corollary is therefore immediate from Corollary \ref{thmvg}.

\begin{corollary}Let $M$ denote the mode of $\overline{Z}$, as defined above. If $n=1,2$, then $M=0$.  If $n\geq3$,
\[\rho\sigma_X\sigma_Y(1-3/n)<M<\rho\sigma_X\sigma_Y(1-2/n).\]
Also, for $n\geq4$,
\[M\geq\frac{\rho\sigma_X\sigma_Y}{2}\bigg[1-\frac{2}{n}+\sqrt{\rho^2\Big(1-\frac{2}{n}\Big)^2+(1-\rho^2)\Big(1-\frac{4}{n}\Big)^2}\bigg],\]
with equality if and only if $n=4$.
\end{corollary}



\subsection{The mode of the McKay Type I distribution}

\begin{theorem}\label{thmmckay} The McKay Type I distribution with density (\ref{seven7}) is unimodal. If $-1/2<m\leq0$, then $M=0$.  If $m>0$, then
\begin{equation}\label{mode1mckay} \frac{(2m-1)bc}{c^2-1}<M<\frac{bc}{c^2-1}\Big[m-1/2+\sqrt{(m+1/2)^2-2m/c^2}\Big]<\frac{2mbc}{c^2-1},
\end{equation}
and we also have that
\begin{equation}\label{exmac}\frac{bc}{c^2-1}<\frac{bc}{c^2-1}\Big[3/2+m-\sqrt{(m+1/2)^2-2m/c^2}\Big]<\mathbb{E}[X]-M<\frac{2bc}{c^2-1}.
\end{equation}
We also have the following alternative lower bound for $M$.  For $m>0$,
\begin{equation}\label{mode5mckay}M>\frac{bc}{c^2-1}\Big[m-1+\sqrt{(m-1)^2+4(c^2-1)m/c^2}\Big].
\end{equation}

The lower and upper bounds in (\ref{mode1mckay}) are best possible amongst all bounds of the form $(2m-\alpha)bc/(c^2-1)$, where $\alpha\in\mathbb{R}$ does not involve the model parameters.  Also, the lower and upper bounds in (\ref{exmac}) are best possible amongst all bounds of the form $\beta bc/(c^2-1)$, where $\beta\in\mathbb{R}$ does not involve the model parameters.
\end{theorem}

\begin{proof}The McKay Type I distribution can be expressed as a sum of two independent gamma random variables \cite[Theorem 3]{ha04}, 
 and is therefore self-decomposable.  Since self-decomposable distributions are unimodal \cite{y78}, it follows that the McKay Type I distribution is unimodal.


If $m=0$, the McKay Type I density is $p(x)=N_{b,c}\mathrm{e}^{-cx/b}I_0(x/b)$, $x>0$, where $N_{b,c}$ is the normalising constant. By the differentiation formula (\ref{ddbi0}) and the inequality $I_1(x)<I_0(x)$ \cite{soni}, we have, for all $x>0$,
\begin{equation*}p'(x)=\frac{N_{b,c}}{b}\mathrm{e}^{-cx/b}\Big(I_1\Big(\frac{x}{b}\Big)-cI_0\Big(\frac{x}{b}\Big)\Big)<\frac{(1-c)N_{b,c}}{b}\mathrm{e}^{-cx/b}I_0\Big(\frac{x}{b}\Big)<0,
\end{equation*}
since $c>1$.  Thus, $M=0$ when $m=0$.  For $-1/2<m<0$, we see from the limiting form (\ref{Itend0}) that the McKay Type I density blows up as $x\downarrow0$.    As $p(x)$ is analytic in $(0,\infty)$ this is the only singularity and so when $-1/2<m<0$ the mode is 0.

Let $m>0$. Recall that the mode satisfies equation (\ref{ystar}). By applying the lower bound of (\ref{sqrtbb}) to (\ref{ystar}) we have that
\begin{equation}\label{inti7}\frac{x/b}{m-1/2+\sqrt{(m+1/2)^2+(x/b)^2}}<\frac{1}{c}.
\end{equation}
On solving inequality (\ref{inti7}) for $x$ we obtain the smallest upper bound in (\ref{mode1mckay}) for $M$, and simple manipulations show that, for $m>0$ and $c>1$,
\[m-1/2+\sqrt{(m+1/2)^2-2m/c^2}<2m.\]
For $m\geq1/2$ we obtain the lower bound of (\ref{mode1mckay}) similarly by instead applying the upper bound of (\ref{sqrtbb}) to (\ref{ystar}) (of course the lower bound of (\ref{mode1mckay}) is also valid for $0<m<1/2$).  As $\mathbb{E}[X]=(2m+1)bc/(c^2-1)$ (see \cite{mckay}), the double inequality (\ref{exmac}) then follows from (\ref{mode1mckay}).  For $m>0$, we obtain inequality (\ref{mode5mckay}) by applying inequality (\ref{segi2}) to (\ref{ystar}) and then solving the resulting inequality to obtain the lower bound for $M$.

With $\phi=bc/(c^2-1)$ equation (\ref{xstar3}) reads
\begin{equation}\label{xstar5}\frac{I_m\big(\frac{cx}{(c^2-1)\phi}\big)}{I_{m-1}\big(\frac{cx}{(c^2-1)\phi}\big)}=\frac{1}{c}.
\end{equation}
Let $m>0$ and fix $\phi$.  Then applying (\ref{Itend0}) to (\ref{xstar5}) gives that, as $c\rightarrow\infty$,
\begin{align*}\frac{1}{2m}\frac{x}{c\phi} \sim \frac{1}{c},
\end{align*}
where we used that $c/(c^2-1)\sim 1/c$, as $c\rightarrow\infty$. Rearranging gives that $x\sim 2m\phi$, as $c\rightarrow\infty$, as required. Now let $m>1/2$. Applying (\ref{roots}) to (\ref{xstar5}) and using that $c^2-1\sim 2(c-1)$, as $c\downarrow1$, gives that, as $c\downarrow1$,
\begin{align*} 0&=c\cdot\frac{\displaystyle 1-\frac{4m^2-1}{4}\cdot\frac{(c-1)\phi}{cx}}{\displaystyle1-\frac{4m^2-1}{4}\cdot\frac{(c-1)\phi}{cx}}-1+O((c-1)^2) \\
&=c-1-\frac{(2m-1)(c-1)\phi}{x}+O((c-1)^2) \\
&=(c-1)\bigg(1-\frac{(2m-1)\phi}{x}\bigg)+O((c-1)^2),
\end{align*}
from which we deduce that  $M= (2m-1)\phi+O(c-1)$, as $c\downarrow1$, as required.
\end{proof}

\begin{remark}Inequality (\ref{mode5mckay}) improves on the lower bound of (\ref{mode1mckay}) if either
\[\quad 1<c<\sqrt{2}, \:0<m<\frac{c^2}{2(2-c^2)},\quad \text{or} \quad c\geq\sqrt{2},\:m>0  .\]
Letting $B_{m,b,c}$ denote the bound in inequality (\ref{mode5mckay}) we have that $B_{m,b,c}\sim 2mb/c$, $c\rightarrow\infty$, which is the same asymptotic behaviour as the upper bounds in (\ref{mode1mckay}).
\end{remark}


\subsection{Further results and comments}\label{sec2.4}

There exists a substantial literature on inequalities for ratios of modified Bessel functions (see \cite{gaunt struve} and references therein), which could in principle be applied to  (\ref{xstar}) and (\ref{ystar}) to obtain alternative bounds for the mode of the GH, VG and McKay Type I distributions (see inequality (\ref{tanhin}) for an example)
that may yield improvements at least in some parameter regimes.
However, whilst the literature contains inequalities for ratios of modified Bessel functions that are more accurate than the inequalities (\ref{seg1}), (\ref{seg2}), (\ref{seg3}), (\ref{sqrtbb}) and (\ref{segi2}) that we have used in this paper (see, for example, the iterative schemes of \cite{rs16,segura} that can be used to obtain more accurate bounds by using the aforementioned bounds as initial bounds in the iterative procedure), when applied to (\ref{xstar}) or (\ref{ystar}) these more accurate bounds tend to lead to inequalities for the mode that are intractable to solve.




 It is in fact a rather fortunate algebraic accident that the inequalities (\ref{seg1}), (\ref{seg2}) and (\ref{seg3}) for $K_{\nu-1}(x)/K_\nu(x)$ yield, through (\ref{xstar}), a tractable inequality for the mode of the $GH(\lambda,\alpha,\beta,\delta,0)$ distribution.  As an illustration, consider the following inequality (see Section 3.2 of \cite{segura})
\[\frac{K_{\nu-1}(x)}{K_{\nu}(x)}<\frac{1/2-\nu+\sqrt{(\nu+1/2)^2+x^2}}{x}, \quad x>0,\:\nu>-1/2,\]
which improves on inequality (\ref{seg2}) for $-1/2<\nu<0$.  Applying this inequality to (\ref{xstar}) yields the inequality
\[\frac{x}{\delta^2+x^2}\Big[1-\lambda+\sqrt{\lambda^2+\alpha^2(\delta^2+x^2)}\Big]<\beta,\]
which can be solved exactly for $x$, but the resulting bound for the mode is (a solution to a quartic equation) is too complicated to be of practical use.  In fact, as mentioned above, most other bounds in the literature for $K_{\nu-1}(x)/K_\nu(x)$ lead to inequalities for $x$ that cannot be solved analytically.


We end this section by recording the special cases in which a closed-form formula is available for the mode of the GH, VG and McKay Type I distributions.

Let $X\sim GH(\lambda,\alpha,\beta,\delta,0)$ and let $M$ denote its mode.  We have already mentioned that that $M=0$ if $\beta=0$.  Let us now note some cases in which $M$ can be calculated explicitly for general $\beta$ (subject to the necessary condition $|\beta|<\alpha$).  Recall that $M$ is the unique solution to (\ref{xstar}). For $n\in\mathbb{Z}$, $K_{n-1/2}(x)/K_{n+1/2}(x)$ is rational function (see formula (\ref{special}), and apply (\ref{parity}) if $n\leq0$), resulting in a simplification of equation (\ref{xstar}).  If $\lambda=1$ or $\lambda=2$ we can apply the formulas in (\ref{special2}) and $K_{-1/2}(x)=\sqrt{\frac{\pi}{2 x}}\mathrm{e}^{-x}$ to obtain simple equations for the mode.  For example, for $\lambda=2$, (\ref{xstar}) reads
\[\frac{x}{1+\alpha\sqrt{\delta^2+x^2}}=\frac{\beta}{\alpha^2},\]
which on rearranging suitably and then squaring both sides leads to a quadratic equation for $M$ (and simple graphical considerations allow us to identify the relevant solution).  Carrying out these calculations leads to the following compact formulas:
\[M=\frac{\beta\delta}{\gamma}, \:(\lambda=1), \quad M=\frac{\beta}{\gamma^2}\Big[1+\sqrt{\beta^2/\alpha^2+\delta^2\gamma^2}\Big], \:(\lambda=2).\]
The case $\lambda=1$ corresponds to the hyperbolic distribution, and the formula is well-known.  For $\lambda=0$ and $\lambda=3$, repeating the same procedure leads to a quartic equation for $M$, the solutions to which are too complicated to be worth repeating here.

Now let $X\sim\mathrm{VG}_2(\lambda,\alpha,\beta,0)$. For $\delta=0$ equation (\ref{xstar}) is simplified and we obtain
\begin{eqnarray}\label{86}M&=&\frac{\beta}{\alpha(\alpha-\beta)}, \quad \lambda=2, \\
\label{87}M&=& \frac{3\beta-\alpha+\sqrt{\alpha^2+6\alpha\beta-3\beta^2}}{2\alpha(\alpha-\beta)}, \quad \lambda=3,
\end{eqnarray}
in addition to $M=0$ for $\lambda\leq1$ (see Corollary \ref{thmvg}).  For $\lambda=4$ and $\lambda=5$ we can again perform simple manipulations to obtain cubic and quartic equations that $M$ satisfies, but, like in the GH case, the formulas are too complicated to present here.  For the $\mathrm{VG}(r,\theta,\sigma,0)$ parametrisation, the formulas (\ref{86}) and (\ref{87}) read
\begin{eqnarray*}M&=&\theta\bigg(1+\frac{1}{\sqrt{1+\kappa}}\bigg), \quad r=4, \\
M &=&\frac{\theta}{2}\bigg(1+\frac{1}{\sqrt{1+\kappa}}\bigg)\bigg(3-\sqrt{1+\kappa}+\sqrt{6\sqrt{1+\kappa}+\kappa-2}\bigg), \quad r=6,
\end{eqnarray*}
where $\kappa=\sigma^2/\theta^2$.

For the McKay Type I distribution, (\ref{ystar}) only simplifies to a tractable equation in the case $m=1/2$ (in which case we use the formulas in (\ref{casei})), and in that case we obtain
\begin{equation*}M=b\tanh^{-1}(1/c).
\end{equation*}
In fact, we have the inequality
\begin{equation}\label{tanhin}M>b\tanh^{-1}(1/c), \quad m>1/2,
\end{equation}
 which follows from applying the inequality $I_\nu(x)/I_{\nu-1}(x)<\mathrm{tanh}(x)$, $x>0$, $\nu>1/2$, \cite{ifantis} to (\ref{ystar}).

\section{Conjectured inequalities and monotonicity results for the median}\label{sec3}
Let $X$ be a random variable with cumulative distribution function $F_X(x)=\mathbb{P}(X\leq x)$, $x\in\mathbb{R}$.  Then the median of $X$ is defined as $\mathrm{Med}(X):=\inf\{x\in\mathbb{R}\,|\,F_X(x)\geq1/2\}$.  In general, it is not possible to obtain an exact formula for the median of the GH, VG or McKay Type I distributions.  However, in addition to the trivial $\beta=0$ case for the GH and VG distributions, the median of the $\mathrm{VG}(2,\theta,\sigma,0)$ distribution (which is an asymmetric Laplace distribution) can be worked out exactly, and this is instructive for what follows. Let $X\sim \mathrm{VG}(2,\theta,\sigma,0)$, which has PDF
\begin{equation*}p_X(x)=\frac{1}{2\sqrt{\theta^2+\sigma^2}}\mathrm{e}^{\theta x/\sigma^2 -\sqrt{\theta^2+\sigma^2} |x|/\sigma^2}, \quad x\in\mathbb{R}.
\end{equation*}
The median of the asymmetric Laplace distribution is well-known \cite{kp01} and given by
\begin{equation}\label{medform}\mathrm{Med}(X)=\big(\theta+\sqrt{\theta^2+\sigma^2}\big)\log\bigg(1+\frac{\theta}{\sqrt{\theta^2+\sigma^2}}\bigg).
\end{equation}
Fix $\theta\in\mathbb{R}$.  Then elementary asymptotic manipulations show that
\begin{align}\label{propmed}\mathrm{Med}(X)\sim 2\theta\log 2, \: \sigma\downarrow0, \quad \text{and} \quad \mathrm{Med}(X)\sim \theta, \: \sigma\rightarrow\infty.
\end{align}

We have been unable to derive bounds for the median of the GH, VG and McKay Type I distributions.  We have, however, been able to make conjectures
concerning bounds for the median. Our conjectures for the McKay Type I and VG distributions are motivated through their representations in terms of sums
and differences of two independent gamma random variables. In this section, we shall describe how we arrived at our conjectures,
and in doing so provide a possible approach to verifying our conjectured bounds for the medians of these distributions. Based on a conjectured lower
bound for the VG distribution, we are able to conjecture a lower bound for the GH distribution, although as the GH distribution with $\delta\not=0$ does
not have a representation as a sum or difference of independent gamma random variables, we have not been able to provide an analogous possible approach to
verifying the conjecture.  We end the section with some numerical results to support our conjectures.

We begin by clarifying} the parametrisation of the gamma distribution used in this paper: for a random variable $G$ with PDF $p_G(x)=\frac{\lambda^r}{\Gamma(r)}x^{r-1}\mathrm{e}^{-\lambda x}$, $x>0$, we write $G\sim\Gamma(r,\lambda)$. Let us now note two useful representations of the VG and McKay Type I distributions.

Let $X_1$ and $X_2$ be independent $\Gamma(r/2,1)$ random variables.  Then, it is almost immediate from \cite[Proposition 1.2, part (vi)]{gaunt vg} that
\begin{equation}\label{vgrep}V_{r,\theta,\sigma}:=(\sqrt{\theta^2+\sigma^2}+\theta)X_1-(\sqrt{\theta^2+\sigma^2}-\theta)X_2\sim\mathrm{VG}(r,\theta,\sigma,0).
\end{equation}
In the limit $\sigma\downarrow0$, $V_{r,\theta,\sigma}$ converges in distribution to a $\Gamma(r/2,(2\theta)^{-1})$ random variable.  It will also be useful to note that $V_{r,\theta,\sigma}=_d \theta V_{r,1,\sqrt{\kappa}}$, where $\kappa=\sigma^2/\theta^2$, which can be seen because we can write $V_{r,\theta,\sigma}=_d\theta[(\sqrt{1+\kappa}+1)X_1-(\sqrt{1+\kappa}-1)X_2]$.



Now, for $m>0$, let $X_1$ and $X_2$ be independent $\Gamma(m+1/2,1)$ random variables.  Let $\phi=bc/(c^2-1)$.  Then, by a re-parametrisation of \cite[Theorem 3]{ha04}, the random variable
\begin{equation}\label{mcrep}Z_{m,c,\phi}:=\frac{\phi(c-1)}{c} X_1+\frac{\phi(c+1)}{c}  X_2
\end{equation}
has the McKay Type I distribution with parameters $m,b,c$.  By basic properties of the gamma distribution, we have that $Z_{m,b,c}\rightarrow_d \Gamma(m+1/2,(2\phi)^{-1})$ as $c\downarrow1$, and $Z_{m,b,c}\rightarrow_d \Gamma(2m+1,1/\phi)$ as $c\rightarrow\infty$.



Since the work of \cite{cr86}, there has been active research into inequalities for the median of the gamma distribution.  We now record some bounds from this literature that will be used in the sequel.  Let $G\sim \Gamma(r,\lambda)$.  Then, it was shown by \cite{bp06} that, for $r>0$ and $\lambda>0$,
\begin{equation*}\frac{r-\log2}{\lambda}<\frac{r}{\lambda}\mathrm{e}^{-\log 2/r}<\mathrm{Med}(G)<\frac{r}{\lambda}\mathrm{e}^{-1/3r}<\frac{1}{\lambda}\bigg(r-\frac{1}{3}+\frac{1}{18r}\bigg).
\end{equation*}
If $r$ is a positive integer, then an alternative bound due to \cite{c94} is available to us:
\begin{equation}\label{gaminteger}\frac{r-1/3}{\lambda}<\mathrm{Med}(G)\leq\frac{r-1+\log 2}{\lambda}, \quad r\in\mathbb{Z}^+,
\end{equation}
with equality in the upper bound for $r=1$. In fact a simple argument can be used to prove that the double inequality (\ref{gaminteger}) holds, with strict inequality, for all $r>1$.  To the best of our knowledge, this result has not previously been stated in the literature.

\begin{theorem}Let $G\sim\Gamma(r,\lambda)$.  Then, for $r>1$ and $\lambda>0$,
\begin{equation*}\frac{r-1/3}{\lambda}<\mathrm{Med}(G)<\frac{r-1+\log2}{\lambda}.
\end{equation*}
\end{theorem}

\begin{proof}Set $\lambda=1$; the general case follows from rescaling. Let $r>1$ and let $\mu(r)$ and $m(r)$ denote the mean and median of a $\Gamma(r,1)$ random variable, respectively.   Corollary 5 of \cite{c94} states that, for $r\in\mathbb{Z}^+$,
\[1-\log2\leq\mu(r)-m(r)<1/3.\]
(A strict inequality is not given in the upper bound in statement of Corollary 5 of \cite{c94}, but from Theorem 1 of the same paper we can see that the inequality is actually strict.)  It was shown by \cite{bp06} that $m'(r)<1$.  Also, $\mu(r)=r$, so that $\mu'(r)=1$.  Therefore $\mu'(r)-m'(r)>0$.  Since $r\mapsto\mu(r)-m(r)$ is a monotone increasing function on $(1,\infty)$, it follows that
\[1-\log2<\mu(r)-m(r)<1/3, \quad \forall r>1.\]
Finally, using that $\mu(r)=r$ and rearranging for $m(r)$ completes the proof.
\end{proof}


We now make two conjectures concerning the median of the sum and difference of independent gamma random variables.

\begin{conjecture}\label{conj1}Let $X_1$ and $X_2$ be independent $\Gamma(r,1)$ random variables.  Define
\begin{equation*}Z_\alpha=\alpha X_1+(2-\alpha)X_2, \quad \alpha>0.
\end{equation*}
Then the function $\alpha\mapsto\mathrm{Med}(Z_\alpha)$ is  non-decreasing for $\alpha\in(0,1)$.
\end{conjecture}

\begin{conjecture}\label{conj2}With the same notation as in Conjecture \ref{conj1}, the function $\alpha\mapsto\mathrm{Med}(Z_\alpha)$ is non-increasing for $\alpha\in(2,\infty)$.  (If Conjecture \ref{conj1} is true, then it follows that
$\alpha\mapsto\mathrm{Med}(Z_\alpha)$ is non-increasing for $\alpha\in(1,2)$, since $Z_{2-\alpha}=_dZ_\alpha$.)
\end{conjecture}


Before using Conjectures \ref{conj1} and \ref{conj2} to formulate conjectures for the median of the VG and McKay Type I distributions,
we present a relation between Schur convexity and Conjectures \ref{conj1} and \ref{conj2}.
In doing so we reduce the problem of proving Conjecture \ref{conj1} to a potentially simpler problem.
The theory of Schur-convexity is very rich and complex, and applicable for deriving various inequalities with multivariate functions.
For what follows it is enough to review necessary definitions in the setup with two variables.
 Given  $\bd{x}, \bd{y} \in \R^2$, we say that
 $\bd{x}=(x_1,x_2)$ is  majorized by $\bd{y}=(y_1,y_2)$ (in notation $\bd{x}\prec\bd{y}$) if and only if $x_1+x_2=y_1+y_2$
 and $\max\{x_1,x_2\}\leq \max\{y_1,y_2\}$. A function
 $\bd{x}\mapsto f(\bd{x})$ is said to be Schur-convex (Schur-concave) on a set $D\subset \R^2$ if
 $\bd{x}\prec \bd{y} \implies f(\bd{x})\leq (\geq) f(\bd{y})$ for all $\bd{x,y}\in D$. More details about Schur convexity and its applications can be found in
the benchmark book \cite{marolk}. The next classical result  may be fundamental for the proof of Conjecture \ref{conj1}.

 \begin{lemma}\cite{bodiac87} \label{bodiac} Let $X_1$ and $X_2$ be independent random variables with $\Gamma (r,1)$ distribution. For
 $\bd{c} =(c_1,c_2)$, where $c_1,c_2>0$, define
\begin{equation} \label{s-func}
  F(\bd{c};t)= \mathbb{P}(c_1X_1+c_2 X_2\leq t).
  \end{equation}
 Then $F$ is  Schur-convex with respect to $\bd{c}$ for $0\leq t\leq  r(c_1+c_2)$ and is Schur-concave for any $t\geq (r+1/2)(c_1+c_2)$.
 \end{lemma}


Let $Z_\alpha$ be as in Conjecture \ref{conj1}. By Lemma \ref{bodiac} with $c_1=\alpha$ and $c_2=2-\alpha$, $\alpha \in (0,1)$, we get that
the function $(\alpha,1-\alpha)\mapsto \mathbb{P}(Z_{\alpha}\leq t)$ is Schur-convex for $t\leq 2r$.
It is easy to see that
\[ (\alpha_1,2-\alpha_1)\prec (\alpha_2,2-\alpha_2) \iff \alpha_1\geq \alpha_2.\]
Hence, by Schur-convexity, we have that
\begin{equation}
\label{hprf}
\alpha_1\geq \alpha_2 \implies \mathbb{P}(Z_{\alpha_1}\leq t)\leq \mathbb{P}(Z_{\alpha_2} \leq t),\quad t\in [0,2r].
\end{equation}
Let $m_1,m_2$ be the medians of $Z_{\alpha_1}$ and $Z_{\alpha_2}$, respectively. From (\ref{hprf}) with $t=m_1$ we get
\[ \alpha_1\geq \alpha_2 \implies  \frac{1}{2} \leq \mathbb{P}(Z_{\alpha_2}\leq m_1) \implies m_2\leq m_1. \]
To complete the proof of Conjecture \ref{conj1}, we have to show  only that $\mathrm{Med}(Z_\alpha)\in [0,2r]$,
that is,  $\mathrm{Med(Z_\alpha)} \leq \mathbb{E} (Z_\alpha)  $ for $\alpha\in (0,1)$.
 Although there is a reasonable literature concerning general methods for proving median-mean inequalities for probability distributions
 (see, for example, \cite{gm77,v79}), we were not able to apply
 the usual sufficient conditions in  the case of the McKay Type I distribution. \hfill $\square$


\medskip

Conjecture \ref{conj2} might be proved using an  approach similar to the one discussed above.
Since $2-\alpha <0$, the second part of Lemma \ref{bodiac} (Schur concavity)
 has to be proved from scratch. Further, it is necessary to  prove that  $\mathrm{Med}(Z_\alpha) \geq 2r+1$ for $\alpha \geq 2$.

\medskip

Through the representations of the VG and McKay Type I distributions in terms of differences and sums of independent gamma random variables we can recast Conjectures \ref{conj1} and \ref{conj2} in terms of conjectured monotonicity results for the medians of these distributions.  In addition, the inequalities for the median of the gamma distribution presented earlier in this section allow us to conjecture tight bounds for the medians of the VG and McKay Type I distributions.

\begin{conjecture}\label{conj3}Let $V_{r,\theta,\sigma}\sim \mathrm{VG}(r,\theta,\sigma,0)$.  Fix $r>0$ and $\theta>0$.  Then $\sigma\mapsto\mathrm{Med}(V_{r,\theta,\sigma})$ is strictly decreasing for $\sigma\in(0,\infty)$. Consequently, $\mathrm{Med}(V_{r,\theta,\sigma})<\mathrm{Med}(G)$, where $G\sim \Gamma(r/2,(2\theta)^{-1})$. In addition we conjecture that $\lim_{\sigma\rightarrow\infty}\mathrm{Med}(V_{r,\theta,\sigma})=(r-1)\theta$ for $r>1$, and $\lim_{\sigma\rightarrow\infty}\mathrm{Med}(V_{r,\theta,\sigma})=0$ for $r\leq1$. We therefore have the following conjectured inequalities:
\begin{equation*}(r-1)\theta<\mathrm{Med}(V_{r,\theta,\sigma})<r\theta \mathrm{e}^{-2/3r}<\bigg(r-\frac{2}{3}+\frac{2}{9r}\bigg)\theta, \quad r>0,
\end{equation*}
and
\begin{equation}\label{concon}\mathrm{Med}(V_{r,\theta,\sigma})\leq(r+2\log2-2)\theta, \quad r\geq2.
\end{equation}
\end{conjecture}







\begin{conjecture}\label{conj4}Let $\phi=bc/(c^2-1)$ and let $Z_{m,c,\phi}$ be a McKay Type I random variable with parameters $m$, $b$ and $c$. Fix $m>-1/2$ and $\phi>0$.  Then $\phi\mapsto\mathrm{Med}(Z_{m,c,\phi})$ is  strictly increasing for $c\in(1,\infty)$.  Consequently, $\mathrm{Med}(G_1)<\mathrm{Med}(Z_{m,c,\phi})<\mathrm{Med}(G_2)$, where $G_1\sim\Gamma(m+1/2,(2\phi)^{-1})$ and $G_2\sim\Gamma(2m+1,1/\phi) $.  Therefore, for $m>-1/2,$
\begin{align*}(2m+1-\log2)\phi<(2m&+1)\phi\mathrm{e}^{-\log2/(2m+1)}<\mathrm{Med}(Z_{m,c,\phi})<\\
&<(2m+1)\phi\mathrm{e}^{-1/3(2m+1)}<\bigg(2m+\frac{2}{3}+\frac{1}{18(2m+1)}\bigg)\phi,
\end{align*}
and, for $m\geq1/2$,
\begin{equation*}(2m+1/3)\phi<\mathrm{Med}(Z_{m,c,\phi})\leq(2m+\log 2)\phi.
\end{equation*}
\end{conjecture}

Phrasing Conjecture \ref{conj3} in the parametrisation (\ref{seven}) of the VG distribution gives the conjectured inequality: for $X\sim\mathrm{VG}_2(\lambda,\alpha,\beta,0)$,
$\mathrm{Med}(X)>\frac{\beta}{\gamma^2}(2\lambda-1)$, $\lambda>1/2$. Based on the functional form of the inequalities for the mode of the GH distribution in Theorem \ref{thmghd}, we generalise this conjectured lower bound for the VG distribution to the GH distribution through in the following conjecture.
We have carried out numerical experiments using \emph{Mathematica} that suggest this conjecture holds.

\begin{conjecture}\label{conj5}
For $X\sim GH(\lambda,\alpha,\beta,\delta,0)$, we conjecture that, for $\lambda>1/2$, and $\beta>0$,
\begin{equation*}\mathrm{Med}(X)>\frac{\beta}{\gamma^2}\Big[\lambda-1/2+\sqrt{(\lambda-1/2)^2+\delta^2\gamma^2}\Big].
\end{equation*}
\end{conjecture}

\begin{remark} (i) The conjectured limit $\lim_{\sigma\rightarrow\infty}\mathrm{Med}(V_{r,\theta,\sigma})=(r-1)\theta$ is seen to be true for $r=2$ by (\ref{propmed}). We also note that the conjectured upper bound (\ref{concon}) is attained in the case $r=2$: $\lim_{\sigma\downarrow0}\mathrm{Med}(V_{2,\theta,\sigma})=2\theta\log2$ (see again (\ref{propmed})).

\vspace{2mm}

\noindent{(ii)} We carried out numerical experiments using \emph{Mathematica} which support the conjectured monotonicity results of Conjectures \ref{conj3} and \ref{conj4} (which in turn support Conjectures \ref{conj1} and \ref{conj2}).
For a given distribution, with specific parameter values, our procedure was to make an initial guess $m_0$ for the median, based on the conjectured inequalities, and then use \emph{Mathematica} to numerically evaluate the integral $I_0=\int_{m_0}^\infty p(x)\,\mathrm{d}x$, where $p(x)$ is the PDF. If $I_0>0.5$, then we make a second guess $m_1 >m_0$, whilst if $I_0<0.5$ we make a second guess $m_1< m_0$. We then numerically evaluate the integral $I_1=\int_{m_1}^\infty p(x)\,\mathrm{d}x$, and iteratively repeat the procedure until an estimate for the median $m$ has been obtained up to our desired accuracy of 3 decimal places.
Some results are reported in Tables \ref{table1} and \ref{table2}.  For testing the conjectures, it suffices to take $\theta=1$ and $\phi=1$, because $\mathrm{Med}(V_{r,\theta,\sigma})=\theta \mathrm{Med}(V_{r,1,\sqrt{\kappa}})$, where $\kappa=\sigma^2/\theta^2$, and $\mathrm{Med}(Z_{m,c,\phi})=\phi \mathrm{Med}(Z_{m,c,1})$ (by (\ref{vgrep}) and (\ref{mcrep})).  We note that Table \ref{table1} also provides support for the conjectured limits $\lim_{\sigma\rightarrow\infty}\mathrm{Med}(V_{r,\theta,\sigma})=(r-1)\theta$ for $r>1$, and $\lim_{\sigma\rightarrow\infty}\mathrm{Med}(V_{r,\theta,\sigma})=0$ for $r\leq1$.  As $r$ decreases down to 1, the limit $\lim_{\sigma\rightarrow\infty}\mathrm{Med}(V_{r,\theta,\sigma})=(r-1)\theta$ is approached quite slowly; for example, $\mathrm{Med}(V_{1.5,1,1000})=0.507$.
\end{remark}

\begin{table}[h]
\centering
\caption{\footnotesize{Median of the $\mathrm{VG}(r,1,\sigma,0)$ distribution.}}
\label{table1}
{\scriptsize
\begin{tabular}{|c|rrrrrr|}
\hline
 \backslashbox{$r$}{$\sigma$}       &    0.1 &    0.3 &    1 & 3 &    10 &    30   \\
 \hline
0.5  & 0.0863 & 0.0798 & 0.0502 & 0.0195 & 0.00582 & 0.00192 \\
1  & 0.454 & 0.444 & 0.380 & 0.276 & 0.198  &  0.157 \\
2.5  & 1.872 & 1.861 & 1.775 & 1.621 & 1.531 & 1.507  \\
5 & 4.350 & 4.338 & 4.246 & 4.084 & 4.012 & 4.001  \\ 
10 & 9.340 & 9.328  & 9.233 & 9.071 & 9.009 & 9.001  \\ 
  \hline
\end{tabular}}
\end{table}

\begin{table}[h]
\centering
\caption{\footnotesize{Median of the McKay Type I distribution with parameters $m$, $b$ and $c$, and $\phi=bc/(c^2-1)=1$.}}
\label{table2}
{\scriptsize
\begin{tabular}{|c|rrrrrr|}
\hline
 \backslashbox{$m$}{$c$}       &    1.01 &    1.04 &    1.2 & 1.8 &    4 &    16   \\
 \hline
0  & 0.458 & 0.466 & 0.517 & 0.619 & 0.679 & 0.692 \\
1  & 2.369 & 2.378 & 2.425 & 2.545 &  2.647 &  2.672 \\
2.5  & 5.351 & 5.361 & 5.408 & 5.526 & 5.637 & 5.668  \\
5 & 10.344 & 10.354 & 10.401 & 10.518 & 10.633 & 10.666  \\ 
10 & 20.341 & 20.350 & 20.397 & 20.514 & 20.630 & 20.665    \\ 
  \hline
\end{tabular}}
\end{table}


\appendix

\section{Modified Bessel functions}\label{appendix}

The following standard properties of the modified Bessel functions  can be found in \cite{olver}.
The modified Bessel functions of the first kind $I_\nu(x)$ and second kind $K_\nu(x)$ are defined, for $\nu\in\mathbb{R}$ and $x>0$, by
\[I_{\nu} (x) =  \sum_{k=0}^{\infty} \frac{(\frac{1}{2}x)^{\nu+2k}}{\Gamma(\nu +k+1) k!} \quad \text{and} \quad K_\nu(x)=\int_0^\infty \mathrm{e}^{-x\cosh(t)}\cosh(\nu t)\,\mathrm{d}t.
\]
For $x>0$, the functions $I_\nu(x)$ and $K_\nu(x)$ are positive for $\nu\geq-1$ and all $\nu\in\mathbb{R}$,  respectively.
We also have that
\begin{align} \label{parity}K_{-\nu}(x)=K_\nu(x), \quad \nu\in\mathbb{R}.
\end{align}
For $\nu=n+1/2$, $n=0,1,2,\ldots$, we have
\begin{equation}\label{special} K_{n+1/2}(x)=\sqrt{\frac{\pi}{2x}}\bigg\{1+\sum_{i=1}^n\frac{(n+i)!}{(n-i)!i!}(2x)^{-i}\bigg\}\mathrm{e}^{-x}.
\end{equation}
In particular,
\begin{equation}\label{special2}K_{1/2}(x)=\sqrt{\frac{\pi}{2x}}\mathrm{e}^{-x}, \quad\!\! K_{3/2}(x)=\sqrt{\frac{\pi}{2x}}\bigg(1+\frac{1}{x}\bigg)\mathrm{e}^{-x}, \quad\!\!  K_{5/2}(x)=\sqrt{\frac{\pi}{2x}}\bigg(1+\frac{3}{x}+\frac{3}{x^2}\bigg)\mathrm{e}^{-x}.
\end{equation}
For the modified Bessel function of the first kind we note the following special cases:
\begin{equation}\label{casei}I_{-1/2}(x)=\sqrt{\frac{2}{\pi x}}\cosh(x), \quad I_{1/2}(x)=\sqrt{\frac{2}{\pi x}}\sinh(x).
\end{equation}
The modified Bessel functions have the following asymptotic behaviour:
\begin{eqnarray}\label{Itend0}I_{\nu} (x) &\sim& \frac{1}{\Gamma(\nu +1)} \left(\frac{x}{2}\right)^{\nu}, \qquad x \downarrow 0, \: \nu>-1, \\
\label{Ktend0}K_{\nu} (x) &\sim& \begin{cases} 2^{|\nu| -1} \Gamma (|\nu|) x^{-|\nu|}, & \quad x \downarrow 0, \: \nu \not= 0, \\
-\log x, & \quad x \downarrow 0, \: \nu = 0, \end{cases} \\
 \label{roots} I_{\nu} (x) &=& \frac{\mathrm{e}^x}{\sqrt{2\pi x}}\bigg(1-\frac{4\nu^2-1}{8x}+O(x^{-2})\bigg), \quad x \rightarrow \infty,\:\nu\in\mathbb{R},  \\
\label{Ktendinfinity} K_{\nu} (x) &=& \sqrt{\frac{\pi}{2x}} \mathrm{e}^{-x}\bigg(1+\frac{4\nu^2-1}{8x}+O(x^{-2})\bigg), \quad x \rightarrow \infty,\: \nu\in\mathbb{R}.
\end{eqnarray}
The following differentiation formulas hold:
\begin{align}
\label{ddbi0}\frac{\mathrm{d}}{\mathrm{d}x}\big( I_0(x)\big)&= I_{1}(x), \\
\label{ddbi}\frac{\mathrm{d}}{\mathrm{d}x}\big(x^\nu I_\nu(x)\big)&=x^{\nu} I_{\nu-1}(x), \\
\label{ddbk}\frac{\mathrm{d}}{\mathrm{d}x}\big(x^\nu K_\nu(x)\big)&=-x^{\nu} K_{\nu-1}(x).
\end{align}

\subsection*{Acknowledgements}
RG is supported by a Dame Kathleen Ollerenshaw Research Fellowship.  MM acknowledges the support by grants III 44006 and 174024 from Ministry of Education, Science and Technological Development of Republic of Serbia.  We would like to thank the referee for his/her helpful comments, which have lead to an improvement in the presentation of the results and organisation of the paper.

\end{document}